\documentclass[11pt, a4paper]{article}
\usepackage{times}
\usepackage{a4wide}
\usepackage{hyperref}
\usepackage[british]{babel}
\usepackage{enumerate, longtable}
\usepackage{amsmath, amscd, amsfonts, amsthm, amssymb, latexsym, comment, graphicx, color, url}
\usepackage[all]{xypic}
\usepackage[T1]{fontenc}
\usepackage[utf8]{inputenc}
\usepackage{verbatim}

\newtheorem{thm}{Theorem}[section]

\newtheorem{defi}[thm]{Definition}

\newcommand{\donothing}[1]{}

\newcommand{\exclude}[1]{}

\begin{document}

\selectlanguage{british}

\title{A Short Note on the Pell-Lucas-Eisenstein Series}
\author{Mine Uysal \footnote{Department of Mathematics, Erzincan Binali Yıldırım University, Faculty of Arts and Sciences, Erzincan, Turkey, mine.uysal@erzincan.edu.tr}, Ilker Inam \footnote{Bilecik Seyh Edebali University, Department of Mathematics, Faculty of Arts and Sciences, 11200 Bilecik, Turkey, ilker.inam@bilecik.edu.tr}, and Engin \"{O}zkan \footnote{Department of Mathematics, Erzincan Binali Yıldırım University, Faculty of Arts and Sciences, Erzincan, Turkey eozkanmath@gmail.com or eozkan@erzincan.edu.tr}}
\maketitle

\begin{abstract}
In this work, we define a new type of Eisenstein-like series by using Pell-Lucas numbers and call them the Pell-Lucas-Eisenstein Series. Firstly, we show that the Pell-Lucas-Eisenstein series are convergent on their domain. Afterwards we prove that they satisfy some certain functional equations. Proofs follows from some on calculations on Pell-Lucas numbers.

\end{abstract}

\section{Introduction and Statement of the Results}

Eisenstein series play a fundamental theory role in the theory of modular forms which brings different branches of mathematics together. More precisely, they serve as important examples with coefficients that are easy to compute even with a very high number of coefficients is requested. For example, in \cite{4}, Eisenstein series are effectively used for obtaining Hecke eigenforms of half-integral weight by Rankin-Cohen brackets. For more details on modular forms see \cite{3} and \cite{8}.

In a recent work in \cite{6}, Matthew Just and Robert Schneider define an Eisenstein-like series summed over integer partitions, and use it to construct families of "semi-modular forms". As a continuation work, in \cite{2}, replacing integer partitions by Fibonacci numbers, Agbolade P. Akande and Robert Schneider construct a family of Eisenstein-like series to produce semi-modular forms as well as Lucas sequences.

Naturally, our inspiration for this work is coming from looking for further examples on partition Eisenstein series. In a recent paper \cite{5}, we define the Pell-Eisenstein Series and give their properties. In this short note, as a follow up work, we define the Pell-Lucas-Eisenstein series.

\begin{defi}
The $n$th-Pell-Lucas number is denoted by $Q_n$ and defined as $Q_0=2$ and $Q_1=2$ and for $n \geq 3$
$$ Q_n = 2 Q_{n-1} + Q_{n-2} .$$
\end{defi}
The reader is referred to \cite{7} for details about Pell-Lucas numbers. A recent work on the topic can be seen at \cite{3}.

\begin{defi}
The Pell-Lucas-Eisenstein Series of weight $m$ is denoted by $Q_m(z)$ defined as

$$ Q_m(z):= \sum_{j=-\infty}^{\infty} (Q_j z+ Q_{j-1})^{-m}  $$
\end{defi}

First of all, we should check whether $Q_m(z)$ is well-defined. 
\begin{thm}
For $m>1$, $Q_m(z)$ Pell-Lucas-Eisenstein series are convergent on its domain.
\end{thm}
\begin{proof}
By Weierstrass criterion, it is easy to see that we can bound $Q_m(z)$ by $\zeta(m)$ and the theorem follows. 
\end{proof}

After a short calculation, one can get the following well-known fact.
\begin{thm}\cite{5}
Let $Q_n$ be the $n$th-Pell-Lucas number. Then $$ Q_{-n}=(-1)^n Q_n .$$
\end{thm}

Now we are ready to state the main contribution of the paper:
\begin{thm}\label{main-thm}
The Pell-Lucas-Eisenstein series are satisfying the following functional equations:
\begin{enumerate}[(i)]
\item $Q_{2k}(\frac{-1}{z})=z^{2k}Q_{2k}(z)$
\item $Q_{2k}(2-z)=Q_{2k}(z)$
\item $Q_{2k}(z+2)=z^{-2k}Q_{2k}(\frac{1}{z})$
\item $Q_{2k}(-z) = z^{-2k} Q_{2k}(\frac{1}{z})$

\end{enumerate}
\end{thm}

\section{Proof of the Main Result}
In this section, we will give proof of the main result, namely Theorem \ref{main-thm}. It will be based on direct calculation.

\begin{proof}
\begin{enumerate}[(i)]
\item Let $Q_{2k}$ be the Pell-Lucas-Eisenstein series of weight $2k$. Then, it is clear that

\[ Q_{2k}\left(\frac{-1}{z}\right) = \sum_{j=-\infty}^{\infty} \left( Q_j \left( \frac{-1}{z} \right) +Q_{j-1}\right)^{-2k}. \]

By expanding the series, we have

\begin{eqnarray*}
Q_{2k}\left(\frac{-1}{z}\right) &=& \dots + \left(\frac{14}{z}+34\right)^{-2k} + \left(\frac{-6}{z}-14\right)^{-2k} + \left(\frac{2}{z}+6\right)^{-2k}  + \\&&  \left(\frac{-2}{z}-2 \right)^{-2k} +\left(\frac{-2}{z}+2 \right)^{-2k}  + \left(-\frac{6}{z}+2 \right)^{-2k}  + \left(-\frac{14}{z}+6\right)^{-2k} + \dots
\end{eqnarray*}

By equating the denominator and taking common brackets, we get

\begin{eqnarray*}  Q_{2k}\left(\frac{-1}{z}\right) &=& z^{2k}[\dots + (34z+14)^{-2k} + (-14z-6)^{-2k} + \\&& (6z+2)^{-2k} + (-2z-2)^{-2k}  + (2z-2)^{-2k} \\&&   + (2z-6)^{-2k} + (6z-14)^{-2k} +  \dots] =z^{2k}Q_{2k}(z)  \end{eqnarray*}

as desired.

\item By the definition of the Pell-Eisenstein series, it is clear that 

\[Q_{2k}(2-z) = \sum_{j=-\infty}^{\infty} (Q_j(2-z)+Q_{j-1})^{-2k} =  \sum_{j=-\infty}^{\infty} (Q_j(z)+Q_{j-1})^{-2k} = Q_{2k}(z). \]

\item Let us define

\[ Q_m(z)=Q_m^- (z)+ Q_m^+(z)\]

where $Q_m^-(z):=\sum_{-\infty < n \leq 0 } (Q_n z + Q_{n-1})^{-m}$ and $Q_m^+(z):=\sum_{1 \leq  n \le \infty } (Q_n z + Q_{n-1})^{-m}$.

It is clear that 
\begin{eqnarray*} Q_{2k}^+(z+2) &=& \sum_{n \geq 1} (Q_n(z+2) + Q_{n-1})^{-2k}\\&=& \sum_{n \geq 1} (Q_n z + 2Q_n + Q_{n-1})^{-2k}
 =\sum_{n \geq 1} (Q_n z + Q_{n+1})^{-2k} .\end{eqnarray*} 

By expanding the series and multiplying by $z^{2k}$ both sides, we have

$$\sum_{n \geq 1} (Q_n z + Q_{n+1})^{-2k}=  \sum_{n \geq 1} (Q_n (\frac{1}{z})+Q_{n-1})^{-2k},$$

hence we get

$$Q_{2k}^+(z+2)= z^{-2k} Q_{2k}^+(\frac{1}{z})- (\frac{2}{z}+2)^{-2k}  .$$ 

Similarly, we have $Q_{2k}^-(z+2)= z^{-2k} Q_{2k}^-(\frac{1}{z})+ (\frac{2}{z}+2)^{-2k}  $.

Finally, by combining two results above, we find

$$Q_{2k}(z+2) = z^{-2k} Q_{2k} (\frac{1}{z}),$$ as desired.

\item Assume the setting above and let us consider

$$Q_{2k}^-(-z)=\sum_{n \leq 0}(Q_n(-z)+Q_{n-1})^{-2k} $$

If we expand the series and afterwards multiplying by $z^{2k}$ then we can easily conclude that

$$Q_{2k}^-(-z)=z^{-2k}Q_{2k}^+(\frac{1}{z}).$$

Exactly the same arguments are valid for the positive part and we get 

$$Q_{2k}^+(-z)=z^{-2k}Q_{2k}^-(\frac{1}{z})$$

The last step will be

$$ Q_{2k}^{-}(-z)+ Q_{2k}^{+}(-z) = z^{-2k} ( Q_{2k}^{+}(1/z)+ Q_{2k}^{-}(1/z) ) $$

and the result follows.

\end{enumerate}
\end{proof}

\section{Conclusion Remarks and Future Work}
In this short note, we introduce the Pell-Lucas-Eisenstein series and study their basic properties. This work gives a positive answer to the theoretical question in \cite{1} about the existence of classes of special functions in another direction. Other arithmetic properties of the Pell-Lucas-Eisenstein series is an open problem for the moment.

\end{document}